\documentclass[reqno]{amsart}
\usepackage{amssymb,setspace}
\usepackage{ifpdf}
\ifpdf
 \usepackage[hyperindex]{hyperref}
\else
 \expandafter\ifx\csname dvipdfm\endcsname\relax
 \usepackage[hypertex,hyperindex]{hyperref}%
 \else
 \usepackage[dvipdfm,hyperindex]{hyperref}%
 \fi
\fi
\allowdisplaybreaks[4]
\numberwithin{equation}{section}
\theoremstyle{plain}
\newtheorem{thm}{Theorem}[section]
\newtheorem{cor}{Corollary}[thm]
\newtheorem{lem}{Lemma}[section]
\theoremstyle{definition}
\newtheorem{dfn}{Definition}[section]
\theoremstyle{remark}
\newtheorem{rem}{Remark}[section]
\DeclareMathOperator{\td}{d\mspace{-2mu}}

\begin{document}

\title[Inequalities of Hermite-Hadamard type]
{Integral inequalities of Hermite-Hadamard type for $\boldsymbol{(\alpha,m)}$-GA-convex functions}

\author[A.-P. Ji]{Ai-Ping Ji}
\address[A.-P. Ji]{College of Mathematics, Inner Mongolia University for Nationalities, Tongliao City, Inner Mongolia Autonomous Region, 028043, China}
\email{\href{mailto: A.-P. Ji <jiaiping999@126.com>}{jiaiping999@126.com}}

\author[T.-Y. Zhang]{Tian-Yu Zhang}
\address[T.-Y. Zhang]{College of Mathematics, Inner Mongolia University for Nationalities, Tongliao City, Inner Mongolia Autonomous Region, 028043, China}
\email{\href{mailto: T.-Y. Zhang <zhangtianyu7010@126.com>}{zhangtianyu7010@126.com}}

\author[F. Qi]{Feng Qi}
\address[F. Qi]{School of Mathematics and Informatics, Henan Polytechnic University, Jiaozuo City, Henan Province, 454010, China; Department of Mathematics, School of Science, Tianjin Polytechnic University, Tianjin City, 300387, China}
\email{\href{mailto: F. Qi <qifeng618@gmail.com>}{qifeng618@gmail.com}, \href{mailto: F. Qi <qifeng618@hotmail.com>}{qifeng618@hotmail.com}, \href{mailto: F. Qi <qifeng618@qq.com>}{qifeng618@qq.com}}
\urladdr{\url{http://qifeng618.wordpress.com}}

\begin{abstract}
In the paper, the authors introduce a notion ``$(\alpha,m)$-GA-convex functions'' and establish some integral inequalities of Hermite-Hadamard type for $(\alpha,m)$-GA-convex functions.
\end{abstract}



\thanks{This work was partially supported by the Foundation of the Research Program of Science and Technology at Universities of Inner Mongolia Autonomous Region under grant number NJZY13159, China}

\thanks{This paper was typeset using \AmS-\LaTeX}

\maketitle

\section{Introduction}

In~\cite{Mihesan-1993-Romania, Toader-Proc-1985-338}, the concepts of $m$-convex functions and $(\alpha,m)$-convex functions were introduced as follows.

\begin{dfn}[\cite{Toader-Proc-1985-338}]
A function $f:[0,b]\to \mathbb{R}$ is said to be $m$-convex for $m\in(0,1]$ if the inequality
\begin{equation}
f(\alpha x+m(1-\alpha)y)\le \alpha f(x)+m(1-\alpha)f(y)
\end{equation}
holds for all $x,y\in [0,b]$ and $\alpha\in[0,1].$
\end{dfn}

\begin{dfn}[\cite{Mihesan-1993-Romania}]
For $f:[0,b]\to\mathbb{R}$ and $(\alpha,m)\in(0,1]^2$, if
\begin{equation}
 f(tx+m(1-t)y)\le t^\alpha f(x)+m(1-t^\alpha)f(y)
\end{equation}
is valid for all $x,y\in[0,b]$ and $t\in[0,1]$, then we say that $f(x)$ is an $(\alpha,m)$-convex function on $[0,b]$.
\end{dfn}

Hereafter, a few of inequalities of Hermite-Hadamard type for the $m$-convex and $(\alpha,m)$-convex functions were presented, some of them can be recited as following theorems.

\begin{thm}[{\cite[Theorems~2.2]{Klaric-Ozdemir-Pecaric-08-1032}}] \label{Klaric-Ozdemir-Pecaric-08-1032}
Let $I\supset\mathbb{R}_0=[0,\infty)$ be an open interval and let $f:I\to\mathbb{R}$ be a differentiable function on $I$ such that $f'\in L([a,b])$ for $0\le a<b<\infty$, where $L([a,b])$ denotes the set of all Lebesgue integrable functions on $[a,b]$. If $|f'(x)|^{q}$ is $m$-convex on $[a,b]$ for some given numbers $m\in(0,1]$ and $q\ge1$, then
\begin{multline}
\biggl|f\biggl(\frac{a+b}{2}\biggr)-\frac{1}{b-a} \int_a^bf(x)\td x\biggr| \\
\le \frac{b-a}{4}
\min \biggl\{\biggl[\frac{|f'(a)|^{q}+m|f'(b/m)|^{q}}{2}\biggr]^{1/q}, \biggl[\frac{m|f'(a/m)|^{q}+|f'(b)|^{q}}{2}\biggr]^{1/q} \biggr\}.
\end{multline}
\end{thm}

\begin{thm}[{\cite[Theorem~3.1]{Klaric-Ozdemir-Pecaric-08-1032}}]
Let $I\supset[0,\infty)$ be an open interval and let $f:I\to(-\infty, \infty)$ be a differentiable function on $I$ such that $f'\in L([a,b])$ for $0\le a<b<\infty$. If $|f'(x)|^{q} $ is $(\alpha,m)$-convex on $[a,b]$ for some given numbers $m, \alpha\in(0,1]$, and $q\ge1$, then
\begin{multline}
\biggl|\frac{f(a)+f(b)}{2}-\frac{1}{b-a} \int_a^bf(x)\td x\biggr|
\le \frac{b-a}{2}\biggl(\frac{1}{2}\biggr)^{1-1/q}\\*
\times\min\biggl\{\biggl[v_1|f'(a)|^{q}+v_{2}m \biggl|f'\biggl(\frac{b}{m}\biggr)\biggr|^{q}\biggr]^{1/q}, \biggl[v_{2}m\biggl|f'\biggl(\frac{a}{m}\biggr)\biggr|^{q} +v_1|f'(b)|^{q}\biggr]^{1/q} \biggr\},
\end{multline}
where
\begin{equation}
v_1=\frac{1}{(\alpha+1)(\alpha+2)}\biggl(\alpha+\frac{1}{2^\alpha}\biggr)
\end{equation}
and
\begin{equation}
v_{2}=\frac{1}{(\alpha+1)(\alpha+2)}\biggl(\frac{\alpha^{2} +\alpha+2}{2}-\frac{1}{2^\alpha}\biggr).
\end{equation}
\end{thm}

For more information on Hermite-Hadamard type inequalities for various kinds of convex functions, please refer to the monograph~\cite{Dragomir-selected-Topic}, the recently published papers~\cite{Hadramard-Convex-Xi-Filomat.tex, H-H-Bai-Wang-Qi-2012.tex, Dragomir-Agarwal-AML-98-95, Kirmaci-AMC-04-146, Kir-Bak-Ozd-Pec-AMC-26-35, Wang-Qi-Xi-Hadamard-IJOPCM.tex, Xi-Bai-Qi-Hadamard-2011-AEQMath.tex}, and closely related references therein.
\par
In this paper, we will introduce a new concept ``$(\alpha,m)$-geometric-arithmetically convex function'' (simply speaking, $(\alpha,m)$-GA-convex function) and establish some integral inequalities of Hermite-Hadamard type for $(\alpha,m)$-GA-convex functions.

\section{A definition and a lemma}

Now we introduce the so-called $(\alpha,m)$-GA-convex functions.

\begin{dfn}\label{(a,m)-GA-convex-dfn}
Let $f:[0,b] \to \mathbb{R}$ and $(\alpha,m)\in [0,1]^2$. If
\begin{equation}\label{(a,m)-GA-convex-dfn-eq}
f\bigl(x^\lambda y^{m(1-\lambda)}\bigr)\le \lambda ^\alpha f(x)+ m(1-\lambda^\alpha )f(y)
\end{equation}
for all $x,y\in [0,b]$ and $\lambda\in [0,1]$, then $f(x)$ is said to be a $(\alpha,m)$-geometric-arithmetically convex function or, simply speaking, an $(\alpha,m)$-GA-convex function. If \eqref{(a,m)-GA-convex-dfn-eq} is reversed, then $f(x)$ is said to be a $(\alpha,m)$-geometric-arithmetically concave function or, simply speaking, a $(\alpha,m)$-GA-concave function.
\end{dfn}

\begin{rem}
When $m=\alpha=1$, the $(\alpha,m)$-GA-convex (concave) function defined in Defintion~\ref{(a,m)-GA-convex-dfn} becomes a GA-convex (concave) function defined in~\cite{Niculescu-MIA-00-155, Niculescu-MIA-03-571}.
\end{rem}

To establish some new Hermite-Hadamard type inequalities for $(\alpha,m)$-GA-convex functions, we need the following lemma.

\begin{lem}\label{lem1-2012-GA-Ji}
Let $f:I \subseteq \mathbb{R}_+=(0,\infty)\to\mathbb{R}$ be a differentiable function and $a,b \in I$ with $a < b$. If $f'(x) \in L([a,b])$, then
\begin{equation}
\frac{b^2f(b)-a^2f(a)}{2}-\int_a^bxf(x)\td x = \frac{\ln b-\ln a}{2}\int_0^1 a^{3(1-t)}b^{3t}f'\bigl(a^{1-t}b^t\bigr)\td t.
\end{equation}
\end{lem}

\begin{proof}
Let $x=a^{1-t}b^t$ for $0\le t\le 1$. Then
\begin{multline*}
(\ln b-\ln a)\int_0^1a^{3(1-t)}b^{3t}f'\bigl(a^{1-t}b^t\bigr)\td t
= \int_a^b x^2f'(x)\td x\\
=b^2f(b)-a^2f(a) -2\int_a^bxf(x)\td x.
\end{multline*}
Lemma~\ref{lem1-2012-GA-Ji} is thus proved.
\end{proof}

\section{Inequalities of Hermite-Hadamard type}

Now we turn our attention to establish inequalities of Hermite-Hadamard type for $(\alpha,m)$-GA-convex functions.

\begin{thm}\label{thm1-2012-GA-Ji}
Let $f:\mathbb{R}_0=[0,\infty)\to\mathbb{R}$ be a differentiable function and $f'\in L([a,b])$ for $0<a<b<\infty $. If $|f'|^q$ is $(\alpha,m)$-GA-convex on $\bigl[0,\max \{a^{1/m},b\}\bigr]$ for $(\alpha,m)\in (0,1]^2$ and $q\ge1$, then
\begin{multline}\label{thm1-2012-GA-Ji-1}
\biggl|\frac{b^2f(b)-a^2f(a)}{2}-\int_a^bxf(x)\td x\biggr|
 \le \frac{\ln b-\ln a}{2}\bigl[L\bigl(a^3,b^3\bigr)\bigr]^{1-1/q} \\
\times \bigl\{ {m\bigl[L\bigl(a^3,b^3\bigr) -G(\alpha,3)\bigr]\bigl|f'\bigl(a^{1/m}\bigr)\bigr|^q + G(\alpha,3)|f'(b)|^q} \bigr\}^{1/q},
\end{multline}
where
\begin{equation}\label{2012-GA-Ji}
G(\alpha,\ell) = \int_0^1 {t^\alpha a^{\ell(1-t)}}b^{\ell t} \td t
\end{equation}
for $\ell \ge 0$ and
\begin{equation}\label{log-mean-dfn-eq}
L(x,y)=\frac{y-x}{\ln y-\ln x}
\end{equation}
for $x,y>0$ with $x\ne y$.
\end{thm}

\begin{proof}
Making use of the $(\alpha,m)$-GA-convexity of $|f'(x)|^q$ on $\bigl[0,\max\bigl\{a^{1/m},b\bigr\}\bigr]$, Lemma~\ref{lem1-2012-GA-Ji}, and H\"older inequality yields
\begin{align*}
&\quad\biggl|\frac{b^2f(b)-a^2f(a)}{2}-\int_a^bxf(x)\td x\biggr| \le \frac{\ln b-\ln a}{2}\int_0^1a^{3(1-t)}b^{3t}\bigl|f'\bigl(a^{1-t}b^t\bigr)\bigr|\td t \\
&\le \frac{\ln b-\ln a}{2}\biggl[\int_0^1a^{3(1-t)}b^{3t}\td t\biggr]^{1-1/q}
\biggl[\int_0^1a^{3(1-t)}b^{3t}\Bigl|f'\Bigl(\bigl(a^{1/m}\bigr)^{m(1-t)}b^t\Bigr)\Bigr|^q\td t\biggr]^{1/q}\\
&\le \frac{\ln b-\ln a}{2}\biggl(\frac{b^3-a^3}{\ln b^3-\ln a^3}\biggr)^{1-1/ q}\\
&\quad\times\biggl[\int_0^1a^{3(1-t)}b^{3t}\bigl(t^\alpha|f'(b)|^q +m(1-t^\alpha)\bigl|f'\bigl(a^{1/m}\bigr)\bigr|^q\bigr)\td t\biggr]^{1/q}\\
&=\frac{\ln b-\ln a}{2}\bigl[L\bigl(a^3,b^3\bigr)\bigr]^{1-1/q}\\
&\quad\times\bigl\{m\bigl[L\bigl(a^3,b^3\bigr)-G(\alpha,3)\bigr] \bigl|f'\bigl(a^{1/m}\bigr)\bigr|^q+G(\alpha,3)|f'(b)|^q\bigr\}^{1/q}.
\end{align*}
As a result, the inequality~\eqref{thm1-2012-GA-Ji-1} follows. The proof of Theorem~\ref{thm1-2012-GA-Ji} is complete.
\end{proof}

\begin{cor}\label{cor-3.1-1-2012-GA-Ji}
Under the conditions of Theorem~\ref{thm1-2012-GA-Ji}, if $q=1$, then
\begin{multline}
\biggl|\frac{b^2f(b)-a^2f(a)}2-\int_a^bxf(x)\td x\biggr|\\
\le \frac{\ln b-\ln a}{2}\bigl\{m\bigl[L\bigl(a^3,b^3\bigr)-G(\alpha,3)\bigr] \bigl|f'\bigl(a^{1/m}\bigr)\bigr|
 + G(\alpha,3)|f'(b)|\bigr\}.
\end{multline}
\end{cor}

\begin{cor}\label{cor-3.1-2-2012-GA-Ji}
Under the conditions of Theorem~\ref{thm1-2012-GA-Ji}, if $\alpha=1$, then
\begin{multline}\label{cor-3.1-2-2012-GA-Ji-1}
\biggl|\frac{b^2f(b)-a^2f(a)}2-\int_a^bxf(x)\td x\biggr|\le\frac{\bigl(b^3-a^3\bigr)^{1-1/q}}{6} \\
\times\Bigl\{m\bigl[L\bigl(a^3,b^3\bigr)-a^3\bigr] \bigl|f'\bigl(a^{1/m}\bigr)\bigr|^q
+ \bigl[b^3-L\bigl(a^3,b^3\bigr)\bigr]|f'(b)|^q\Bigr\}^{1/q}.
\end{multline}
\end{cor}

\begin{proof}
This follows from the fact that
\begin{equation*}
G(1,3)=\int_0^1ta^{3(1-t)}b^{3t}\td t =\frac{b^3-L\bigl(a^3,b^3\bigr)}{3(\ln b-\ln a)}.
\end{equation*}
The proof of Corollary~\ref{cor-3.1-2-2012-GA-Ji} is complete.
\end{proof}

\begin{cor}\label{cor-3.1-3-2012-GA-Ji}
Under the conditions of Theorem~\ref{thm1-2012-GA-Ji}, we have
\begin{multline}\label{cor-3.1-3pineq1}
\biggl|\frac{b^2f(b)-a^2f(a)}2-\int_a^bxf(x)\td x\biggr|
\le \frac{\ln b-\ln a}{2}\bigl[L\bigl(a^3,b^3\bigr)\bigr]^{1-1/q} \\
\times\biggl(\frac{1}{\alpha+1}\biggr)^{1/q}
\bigl\{m\bigl[(\alpha+1)L\bigl(a^3,b^3\bigr)-b^3\bigr] \bigl|f'\bigl(a^{1/m}\bigr)\bigr|^q+b^3|f'(b)|^q\bigr\}^{1/q}
\end{multline}
and
\begin{equation}\label{cor-3.1-3pineq2}
\biggl|\frac{b^2f(b)-a^2f(a)}2-\int_a^bxf(x)\td x\biggr|
\le \frac{\ln b-\ln a}{2}L\bigl(a^3,b^3\bigr)|f'(b)|.
\end{equation}
\end{cor}

\begin{proof}
Using $\bigl(\frac{b}a\bigl)^{3t}\le\bigl(\frac{b}a\bigl)^3$ for $t\in[0,1]$ in~\eqref{2012-GA-Ji} gives
\begin{equation*}
G(\alpha,3)=a^3\int_0^1t^\alpha\biggl(\frac{b}a\biggr)^{3t}\td t\le\frac{b^3}{\alpha+1}.
\end{equation*}
Substituting this inequality into~\eqref{thm1-2012-GA-Ji-1} yields~\eqref{cor-3.1-3pineq1}.
\par
Utilizing $¡°t^\alpha\le1$ for $t\in[0,1]$ in~\eqref{2012-GA-Ji} reveals
\begin{equation*}
G(\alpha,3)\le \int_0^1a^{3(1-t)}b^{3t}\td t=L\bigl(a^3,b^3\bigr).
\end{equation*}
Combining this inequality with~\eqref{thm1-2012-GA-Ji-1} yields~\eqref{cor-3.1-3pineq2}.
Corollary \ref{cor-3.1-3-2012-GA-Ji} is thus proved.
\end{proof}

\begin{thm}\label{thm2-2012-GA-Ji}
Let $f:\mathbb{R}_0\to\mathbb{R}$ be a differentiable function and $f'\in L([a,b])$ with $0<a<b<\infty $. If $|f'|^q$ is $(\alpha,m)$-GA-convex on $\bigl[0,\max\bigl\{a^{1/m},b\bigr\}\bigr]$ for $(\alpha,m)\in (0,1]^2$ and $q>1$, then
\begin{multline}
\biggl|\frac{b^2f(b)-a^2f(a)}{2}-\int_a^bxf(x)\td x\biggr| \le \frac{\ln b-\ln a}{2}\biggl(\frac{1}{\alpha+1}\biggr)^{1/q}\\
\times\bigl[L\bigl(a^{3q/(q-1)},b^{3q/(q-1)}\bigr)\bigr]^{1-1/q}
\bigl[|f'(b)|^q+\alpha m\bigl|f'\bigl(a^{1/m}\bigr)\bigr|^q\bigr]^{1/q},
\end{multline}
where $L$ is defined by~\eqref{log-mean-dfn-eq}.
\end{thm}

\begin{proof}
Since $|f'(x)|^q$ is $(\alpha,m)$-GA-convex on $\bigl[0,\max\bigl\{a^{1/m},b\bigr\}\bigr]$, from Lemma~\ref{lem1-2012-GA-Ji} and H\"older inequality, we have
\begin{align*}
&\quad\biggl|\frac{b^2f(b)-a^2f(a)}{2}-\int_a^bxf(x)\td x\biggr|
\le \frac{\ln b-\ln a}{2}\int_0^1a^{3(1-t)}b^{3t}\bigl|f'\bigl(a^{1-t}b^t\bigr)\bigr|\td t \\
&\le \frac{\ln b-\ln a}{2}\biggl[\int_0^1a^{3q/(q-1)(1-t)}b^{3q/(q-1)t}\td t\biggr]^{1-1/q}
\biggl[\int_0^1\Bigl|f'\Bigl(\bigl(a^{1/m}\bigr)^{m(1-t)}b^t\Bigr)\Bigr|^q\td t\biggr]^{1/q} \\
&\le \frac{\ln b-\ln a}{2}\biggl[\frac{b^{3q/(q-1)}-a^{3q/(q-1)}}{\ln b^{3q/(q-1)}
 -\ln a^{3q/(q-1)}}\biggr]^{1-1/q}\\
&\quad\times\biggl[\int_0^1\bigl(t^\alpha |f'(b)|^q+m(1-t^\alpha)\bigl|f'\bigl(a^{1/m}\bigr)\bigr|^q\bigr)\td t\biggl]^{1/q} \\
&= \frac{\ln b-\ln a}{2}\bigl[L\bigl(a^{3q/(q-1)},b^{3q/(q-1)}\bigr)\bigr]^{1-1/q}
\biggl[\frac{1}{\alpha + 1}|f'(b)|^q + \frac{\alpha m}{\alpha+1}\bigl|f'\bigl(a^{1/m}\bigr)\bigr|^q\biggr]^{1/q}.
\end{align*}
 The proof of Theorem~\ref{thm2-2012-GA-Ji} is complete.
\end{proof}

\begin{cor}\label{cor-3.2-2012-GA-Ji}
Under the conditions of Theorem~\ref{thm2-2012-GA-Ji}, if $\alpha=1$, then
\begin{multline}
\biggl|\frac{b^2f(b)-a^2f(a)}2-\int_a^bxf(x)\td x\biggr| \le \frac{\ln b-\ln a} {2^{1+1/q}}\\
\times\bigl[L\bigl(a^{3q/(q-1)}, b^{3q/(q-1)}\bigr)\bigr]^{1-1/q}\bigl[|f'(b)|^q +m\bigl|f'\bigl(a^{1/m}\bigr)\bigr|^q\bigr]^{1/q}.
\end{multline}
\end{cor}

\begin{thm}\label{thm3-1-2012-GA-Ji}
Let $f:\mathbb{R}_0\to\mathbb{R}$ be a differentiable function and $f'\in L([a,b])$ for $0<a<b<\infty $. If $|f'|^q$ is $(\alpha,m)$-GA-convex on $\bigl[0,\max\bigl\{a^{1/m},b\bigr\}\bigr]$ for $q>1$ and $(\alpha,m)\in (0,1]^2$, then
\begin{multline}
\biggl|\frac{b^2f(b)-a^2f(a)}{2}-\int_a^bxf(x)\td x\biggr| \le \frac{\ln b-\ln a}{2} \\
\times\bigl\{m\bigl[L\bigl(a^{3q},b^{3q}\bigr)-G(\alpha,3q)\bigr] \bigl|f'\bigl(a^{1/m}\bigr)\bigr|^q+G(\alpha,3q)|f'(b)|^q\bigr\}^{1/q},
\end{multline}
where $G$ and $L$ are respectively defined by~\eqref{2012-GA-Ji} and~\eqref{log-mean-dfn-eq}.
\end{thm}

\begin{proof}
Since $|f'(x)|^q$ is $(\alpha,m)$-GA-convex on $\bigl[0,\max\bigl\{a^{1/m},b\bigr\}\bigr]$, from Lemma~\ref{lem1-2012-GA-Ji} and H\"older inequality, we have
\begin{align*}
&\quad\biggl|\frac{b^2f(b)-a^2f(a)}{2}-\int_a^bxf(x)\td x\biggr| \\
&\le\frac{\ln b-\ln a}{2}\biggl(\int_0^11\td t\biggr)^{1-1/q}
\biggl[\int_0^1a^{3q(1-t)}b^{3qt}\Bigl|f'\Bigl(\bigl(a^{1/m}\bigr)^{m(1-t)}b^t\Bigr)\Bigr|^q\td t\biggr]^{1/q} \\
&\le\frac{\ln b-\ln a}{2}\bigl[mL\bigl(a^{3q},b^{3q}\bigr) \bigl|f'\bigl(a^{1/m}\bigr)\bigr|^q+G(\alpha,3q)
\bigl(|f'(b)|^q-m\bigl|f'\bigl(a^{1/m}\bigr)\bigr|^q\bigr)\bigr]^{1/q}.
\end{align*}
 The proof of Theorem~\ref{thm3-1-2012-GA-Ji} is complete.
\end{proof}

\begin{cor}\label{cor-3-1.3-2012-GA-Ji}
Under the conditions of Theorem~\ref{thm3-1-2012-GA-Ji}, if $\alpha=1$, then
\begin{multline}
\biggl|\frac{b^2f(b)-a^2f(a)}{2}-\int_a^bxf(x)\td x\biggr|
\le\frac{(\ln b-\ln a)^{1-1/q}}{2}\biggl(\frac{1}{3q}\biggr)^{1/q} \\
\times \bigl\{m\bigl[L\bigl(a^{3q},b^{3q}\bigr)-a^{3q}\bigr]\bigl|f'\bigl(a^{1/m}\bigr)\bigr|^q
+ \bigl[b^{3q}-L\bigl(a^{3q},b^{3q}\bigr)\bigr]|f'(b)|^q\bigr\}^{1/q}.
\end{multline}
\end{cor}

\begin{proof}
From
\begin{equation*}
G(1,3q) = \int_0^1 ta^{3q(1-t)}b^{3qt} \td t
= \frac{b^{3q}-L\bigl(a^{3q},b^{3q}\bigr)}{\ln b^{3q}-\ln a^{3q}},
\end{equation*}
Corollary \ref{cor-3-1.3-2012-GA-Ji} follows.
\end{proof}

\begin{thm}\label{thm3-2012-GA-Ji}
Let $f:\mathbb{R}_0\to\mathbb{R}$ be a differentiable function and $f'\in L([a,b])$ for $0<a<b<\infty $. If $|f'|^q$ is $(\alpha,m)$-GA-convex on $\bigl[0,\max\bigl\{a^{1/m},b\bigr\}\bigr]$ for $q>1$, $q> p>0$, and $(\alpha,m)\in (0,1]^2$, then
\begin{multline}
\biggl|\frac{b^2f(b)-a^2f(a)}{2}-\int_a^bxf(x)\td x\biggr|
\le \frac{\ln b-\ln a}{2}\bigl[L\bigl(a^{3(q-p)/(q-1)},b^{3(q-p)/(q-1)}\bigr)\bigr]^{1-1/q} \\
\times \bigl\{m\bigl[L\bigl(a^{3p},b^{3p}\bigr)-G(\alpha,3p)\bigr]\bigl|f'\bigl(a^{1/m}\bigr)\bigr|^q +G(\alpha,3p)|f'(b)|^q\bigr\}^{1/q},
\end{multline}
where $G$ and $L$ are respectively defined by~\eqref{2012-GA-Ji} and~\eqref{log-mean-dfn-eq}.
\end{thm}

\begin{proof}
Since $|f'(x)|^q$ is $(\alpha,m)$-GA-convex on $\bigl[0,\max\bigl\{a^{1/m},b\bigr\}\bigr]$, from Lemma~\ref{lem1-2012-GA-Ji} and H\"older inequality, we have
\begin{align*}
&\quad\biggl|\frac{b^2f(b)-a^2f(a)}{2}-\int_a^bxf(x)\td x\biggr| \\
&\le \frac{\ln b-\ln a}{2}\biggl[\int_0^1a^{3(q-p)/(q-1)(1-t)}b^{3(q-p)/(q-1)t}\td t\biggr]^{1-1/q} \\
&\quad\times\biggl[\int_0^1a^{3p(1-t)}b^{3pt} \Bigl|f'\Bigl(\bigl(a^{1/m}\bigr)^{m(1-t)}b^t\Bigr)\Bigr|^q\td t\biggr]^{1/q} \\
 &\le \frac{\ln b-\ln a}{2}\biggl[\frac{b^{3(q-p)/(q-1)}-a^{3(q-p)/(q-1)}}{\ln b^{3(q-p)/(q-1)}
 -\ln a^{3(q-p)/(q-1)}}\biggr]^{1-1/q} \\
&\quad\times\biggl[\int_0^1a^{3p(1-t)}b^{3pt}\bigl(t^\alpha |f'(b)|^q
+m(1-t^\alpha )\bigl|f'\bigl(a^{1/m}\bigr)\bigr|^q\bigr)\td t\biggr]^{1/q} \\
 &=\frac{\ln b-\ln a}{2}\bigl[L\bigl(a^{3(q-p)/(q-1)},b^{3(q-p)/(q-1)}\bigr)\bigr]^{1-1/q} \\
&\quad\times \bigl[mL\bigl(a^{3p},b^{3p}\bigr)\bigl|f'\bigl(a^{1/m}\bigr)\bigr|^q+G(\alpha,3p)
\bigl(|f'(b)|^q-m\bigl|f'\bigl(a^{1/m}\bigr)\bigr|^q\bigr)\bigr]^{1/q}.
\end{align*}
 The proof of Theorem~\ref{thm3-2012-GA-Ji} is complete.
\end{proof}

\begin{cor}\label{cor-3.3-2012-GA-Ji}
Under the conditions of Theorem~\ref{thm3-2012-GA-Ji}, if $\alpha=1$, then
\begin{multline}
\biggl|\frac{b^2f(b)-a^2f(a)}{2}-\int_a^bxf(x)\td x\biggr| \le \frac{(\ln b-\ln a)^{1-1/q}}{2}\biggl(\frac{1}{3p}\biggr)^{1/q}\\
\times\bigl[L\bigl(a^{3(q-p)/(q-1)},b^{3(q-p)/(q-1)}\bigr)\bigr]^{1-1/q} \\
\times \bigl\{m\bigl[L\bigl(a^{3p},b^{3p}\bigr)-a^{3p}\bigr]\bigl|f'\bigl(a^{1/m}\bigr)\bigr|^q
+ \bigl[b^{3p}-L\bigl(a^{3p},b^{3p}\bigr)\bigr]|f'(b)|^q\bigr\}^{1/q}.
\end{multline}
\end{cor}

\begin{proof}
By
\begin{equation*}
G(1,3p) = \int_0^1 ta^{3p(1-t)}b^{3pt} \td t
= \frac{b^{3p}-L\bigl(a^{3p},b^{3p}\bigr)}{\ln b^{3p}-\ln a^{3p}},
\end{equation*}
Corollary \ref{cor-3.3-2012-GA-Ji} can be proved easily.
\end{proof}

\begin{thm}\label{thm4-2012-GA-Ji}
Let $f, g:\mathbb{R}_0\to\mathbb{R}_0$ and $fg\in L([a,b])$ for $0<a<b<\infty $. If $f^q(x)$ is $(\alpha_1 ,m_1)$-GA-convex on $\bigl[0,\max\bigl\{a^{1/m_1},b\bigr\}\bigr]$ and $g^q(x)$ is $(\alpha_2,m_2)$-GA-convex on $\bigl[0,\max\bigl\{a^{1/m_2},b\bigr\}\bigr]$ for $q\ge 1$, $(\alpha_1,m_1)$, and $(\alpha_2,m_2)\in (0,1]^2$, then
\begin{multline}
\int_a^bf(x)g(x)\td x \le (\ln b-\ln a)[L(a,b)]^{1-1/q}\bigl\{m_1m_2[L(a,b) -G(\alpha_1,1)-G(\alpha_2,1)\\
+G(\alpha_1+\alpha_2,1)]f^q \bigl(a^{1/m_1}\bigr)g^q \bigl(a^{1/m_2}\bigr)
+m_1[G(\alpha_2,1)-G(\alpha_1+\alpha_2,1)]f^q\bigl(a^{1/m_1}\bigr)g^q(b)\\
+ m_2 [G(\alpha_1,1)-G(\alpha_1+\alpha_2 ,1)]f^q(b)g^q \bigl(a^{1/m_2}\bigr)
+G(\alpha_1+\alpha_2,1)f^q(b)g^q(b)\bigr\}^{1/q},
\end{multline}
where $G$ and $L$ are respectively defined by~\eqref{2012-GA-Ji} and~\eqref{log-mean-dfn-eq}.
\end{thm}

\begin{proof}
Using the $(\alpha_1 ,m_1)$-GA-convexity of $f^q(x)$ and the $(\alpha_2,m_2)$-GA-convexity of $g^q(x)$, we have
\begin{equation*}
f^q\bigl(a^{1-t}b^t\bigr)\le t^{\alpha_1}f^q(b)+m_1(1-t^{\alpha_1})f^q\bigl(a^{1/m_1}\bigr)
\end{equation*}
and
\begin{equation*}
g^q\bigl(a^{1-t}b^t\bigr)\le t^{\alpha_2}g^q(b)+m_2(1-t^{\alpha_2})g^q\bigl(a^{1/m_2}\bigr)
\end{equation*}
for $0\le t\le 1$. Letting $x=a^{1-t}b^t$ for $0\le t\le 1$ and using H\"older's inequality figure out
\begin{align*}
&\quad\int_a^bf(x)g(x)\td x
=(\ln b-\ln a)\int_0^1a^{1-t}b^tf\bigl(a^{1-t}b^t\bigr)g\bigl(a^{1-t}b^t\bigr)\td t \\
&\le (\ln b-\ln a)\biggl(\int_0^1a^{1-t}b^t\td t\biggr)^{1-1/q}
 \biggl\{\int_0^1a^{1-t}b^t\bigl[f\bigl(a^{1-t}b^t\bigr)g\bigl(a^{1-t}b^t\bigr)\bigr]^q\td t\biggr\}^{1/q}\\
&\le (\ln b-\ln a)\biggl(\int_0^1a^{1-t}b^t\td t\biggr)^{1-1/q}\biggl\{\int_0^1a^{1-t}b^t\bigl[t^{\alpha _1}f^q(b)\\
&\quad+m_1(1-t^{\alpha_1})f^q\bigl(a^{1/m_1}\bigr)\bigr]
\bigl[t^{\alpha_2}g^q(b)+m_2(1-t^{\alpha_2})g^q\bigl(a^{1/m_2}\bigr)\bigr]\td t\biggr\}^{1/q}\\
&=(\ln b-\ln a)[L(a,b)]^{1- 1/q}\biggl\{\int_0^1a^{1-t}b^t\bigl[t^{\alpha_1+\alpha_2}f^q(b)g^q(b)\\
&\quad+m_1 t^{\alpha_2 }(1-t^{\alpha_1 })f^q\bigl(a^{1/m_1}\bigr)g^q(b)
+ m_2t^{\alpha_1} (1-t^{\alpha _2 })f^q(b)g^q\bigl(a^{1/m_2}\bigr)\\
&\quad+ m_1m_2(1-t^{\alpha_1})(1-t^{\alpha _2})f^q\bigl(a^{1/m_1}\bigr)g^q\bigl(a^{1/m_2}\bigr)\bigr]\td t \biggr\}^{1/q} \\
&=(\ln b-\ln a)[L(a,b)]^{1-1/q}\bigl\{m_1m_2[L(a,b)-G(\alpha_1,1) \\
&\quad-G(\alpha_2,1)+G(\alpha_1+\alpha_2,1)]f^q\bigl(a^{1/m_1}\bigr)g^q\bigl(a^{1/m_2}\bigr)\\
&\quad +m_1[G(\alpha_2,1)-G(\alpha_1+\alpha_2,1)]f^q\bigl(a^{1/m_1}\bigr)g^q(b)\\
&\quad+ m_2 [G(\alpha_1,1)-G(\alpha_1+\alpha_2 ,1)]f^q(b)g^q\bigl(a^{1/m_2}\bigr) +G(\alpha_1+\alpha_2,1)f^q(b)g^q(b)\bigr\}^{1/q}.
\end{align*}
The proof of Theorem~\ref{thm4-2012-GA-Ji} is complete.
\end{proof}

\begin{cor}\label{cor-3.4-1-2012-GA-Ji-cor}
Under the conditions of Theorem~\ref{thm4-2012-GA-Ji},
\begin{enumerate}
\item
if $q=1$, then
\begin{multline}
\int_a^bf(x)g(x)\td x \le(\ln b-\ln a)\bigl\{m_1m_2[L(a,b)-G(\alpha_1,1)-G(\alpha_2,1)\\
+G(\alpha_1+\alpha_2,1)]f\bigl(a^{1/m_1}\bigr)g\bigl(a^{1/m_2}\bigr)
+m_1[G(\alpha_2,1)-G(\alpha_1+\alpha_2,1)]f\bigl(a^{1/m_1}\bigr)g(b) \\
+ m_2 [G(\alpha_1,1)-G(\alpha_1+\alpha_2 ,1)]f(b)g\bigl(a^{1/m_2}\bigr)
+G(\alpha_1+\alpha_2,1)f(b)g(b) \bigr\},
\end{multline}
\item
if $q=1$ and $\alpha_1=\alpha_2=m_1=m_2=1$, then
\begin{multline}
\int_a^bf(x)g(x)\td x\le\frac{1}{\ln b-\ln a}\{[2L(a,b)-a(\ln b-\ln a)-2a]f(a)g(a)+[a+b\\
- 2L(a,b)][f(a)g(b)+f(b)g(a)]+[2L(a,b)+b(\ln b-\ln a)-2b]f(b)g(b)\},
\end{multline}
\item
if $\alpha_1=\alpha_2=m_1=m_2=1$, then
\begin{multline}
\int_a^bf(x)g(x)\td x\le \frac{[L(a,b)]^{1-1/q}}{(\ln b-\ln a)^{2/(q-1)}}\bigl\{[2L(a,b)-a(\ln b-\ln a)-2a]f^q(a)g^q(a)\\
+[a+b- 2L(a,b)][f^q(a)g^q(b)+f^q(b)g^q(a)]\\
+[2L(a,b)+b(\ln b-\ln a)-2b]f^q(b)g^q(b)\bigr\}^{1/q}.
\end{multline}
\end{enumerate}
\end{cor}

\begin{thm}\label{thm5-2012-GA-Ji}
Let $f, g:\mathbb{R}_0\to\mathbb{R}_0$ and $fg\in L([a,b])$ for $0<a<b<\infty $. If
$f^q(x)$ is $(\alpha_1 ,m_1)$-GA-convex on $\bigl[0,\max\bigl\{a^{1/m_1},b\bigr\}\bigr]$ and $g^{q/(q-1)}(x)$ is $(\alpha_2,m_2)$-GA-convex on $\bigl[0,\max\bigl\{a^{1/m_2},b\bigr\}\bigr]$ for $q>1$, $(\alpha_1,m_1)$, and $(\alpha_2,m_2)\in (0,1]^2$, then
\begin{multline}\label{thm5-2012-GA-Ji-1-eq}
\int_a^bf(x)g(x)\td x \le (\ln b-\ln a)\bigl\{m_1f^q\bigl(a^{1/m_1}\bigr)L(a,b)\\
+ G(\alpha_1,1)\bigl[f^q(b)-m_1f^q\bigl(a^{1/ m_1}\bigr)\bigr]\bigr\}^{1 /q}\bigl\{m_2g^{q/(q-1)}\bigl(a^{1/m_2}\bigr)L(a,b)\\
+ G(\alpha _2 ,1)\bigl[g^{q/(q-1)}(b) -m_2 g^{q / (q-1)}\bigl(a^{1/m_2}\bigr)\bigr]\bigr\}^{1-1/q},
\end{multline}
where $G$ and $L$ are respectively defined by~\eqref{2012-GA-Ji} and~\eqref{log-mean-dfn-eq}.
\end{thm}

\begin{proof}
By the $(\alpha_1,m_1)$-GA-convexity of $f^q(x)$ and the $(\alpha_2,m_2)$-GA-convexity of $g^{q/(q-1)}(x)$, we have
\begin{equation*}
f^q\bigl(a^{1-t}b^t\bigr)\le t^{\alpha_1}f^q(b)+m_1(1-t^{\alpha_1})f^q\bigl(a^{1/m_1}\bigr)
\end{equation*}
and
\begin{equation*}
g^{q/(q-1)}\bigl(a^{1-t}b^t\bigr)\le t^{\alpha_2}g^{q/(q-1)}(b)+m_2(1-t^{\alpha_2})g^{q/(q-1)}\bigl(a^{1/m_2}\bigr)
\end{equation*}
for $t\in [0,1]$. Letting $x=a^{1-t}b^t$ for $0\le t\le 1$ and employing H\"older's inequality yield
\begin{align*}
&\quad\int_a^bf(x)g(x)\td x\le \biggl[\int_a^bf^q(x)\td x\biggr]^{1/q}
\biggl[\int_a^b g^{q/(q-1)}(x)\td x \biggr]^{1-1/q}\\
 &=(\ln b-\ln a)\biggl[\int_0^1a^{1-t}b^tf^q\bigl(a^{1-t}b^t\bigr)\td t \biggr]^{1 /q}
 \biggl[\int_0^1a^{1-t}b^tg^{q/(q-1)}\bigl(a^{1-t}b^t\bigr)\td t\biggr]^{1-1/q} \\
&\le (\ln b-\ln a)\biggl[\int_0^1 a^{1-t}b^t\bigl[t^{\alpha_1}
f^q(b)+ m_1(1-t^{\alpha_1} )f^q\bigl(a^{1/m_1}\bigr)\bigr]\td t\biggr]^{1/q} \\
&\quad\times \biggl[\int_0^1a^{1-t}b^t\bigl[t^{\alpha_2}g^{q/(q-1)}(b)+ m_2(1-t^{\alpha_2})g^{q/(q-1)}\bigl(a^{1/m_2}\bigr)\bigr]\td t \biggr]^{1-1/q} \\
&=(\ln b-\ln a)\bigl\{m_1f^q\bigl(a^{1/m_1}\bigr)L(a,b)+G(\alpha_1,1)\bigl[f^q(b)- m_1f^q\bigl(a^{1/m_1}\bigr)\bigr]\bigr\}^{1/q}\\
&\quad\times \bigl\{m_2 g^{q/(q-1)}\bigl(a^{1/m_2}\bigr)L(a,b)\\
&\quad+G(\alpha _2 ,1) \bigl[g^{q/(q-1)}(b)-m_2g^{q/(q-1)}\bigl(a^{1/m_2}\bigr)\bigr]\bigr\}^{1- 1/q}.
\end{align*}
The proof of Theorem~\ref{thm5-2012-GA-Ji} is complete.
\end{proof}

\begin{cor}\label{cor-3.5-2012-GA-Ji}
Under the conditions of Theorem~\ref{thm5-2012-GA-Ji}, if $\alpha_1=\alpha_2=m_1=m_2=1$, then
\begin{multline}
\int_a^bf(x)g(x)\td x\le \{f^q(a)[L(a,b)-a]+[b-L(a,b)]f^q(b)\}^{1/q} \\
\times\bigl\{g^{q/(q-1)}(a)[L(a,b)-a]+[b-L(a,b)]g^{q/(q-1)}(b)\bigr\}^{1-1/q}.
\end{multline}
\end{cor}

\begin{thm}\label{thm6-2012-GA-Ji}
Let $f, g:\mathbb{R}_0\to\mathbb{R}_0$ and $fg\in L([a,b])$ for $0<a<b<\infty $. If
$f(x)$ is $(\alpha_1 ,m_1)$-GA-concave on $\bigl[0,\max\bigl\{a^{1/m_1},b\bigr\}\bigr]$ and $g(x)$ is $(\alpha_2,m_2)$-GA-concave on $\bigl[0,\max\bigl\{a^{1/m_2},b\bigr\}\bigr]$ for $(\alpha_1,m_1)\in (0,1]^2$ and $(\alpha_2,m_2)\in (0,1]^2$, then
\begin{multline}\label{thm5-2012-GA-Ji-2}
\int_a^bf(x)g(x)\td x
\ge(\ln b-\ln a)\bigl\{m_1m_2[L(a,b)-G(\alpha_1 ,1)-G(\alpha_2,1) \\
+G(\alpha_1+\alpha_2 ,1)]f\bigl(a^{1/m_1}\bigr)g\bigl(a^{1/m_2}\bigr)
 +m_1[G(\alpha_2,1)-G(\alpha_1+\alpha_2,1)]f\bigl(a^{1/m_1}\bigr)g(b) \\
+m_2[G(\alpha_1,1)-G(\alpha_1+\alpha_2,1)]f(b)g\bigl(a^{1/m_2}\bigr)+G(\alpha_1+\alpha_2,1)f(b)g(b)\bigr\},
\end{multline}
where $G$ and $L$ are respectively defined by~\eqref{2012-GA-Ji} and~\eqref{log-mean-dfn-eq}.
\end{thm}

\begin{proof}
Since $f(x)$ is $(\alpha_1 ,m_1)$-GA-concave
on $\bigl[0,\max\bigl\{a^{1/m_1},b\bigr\}\bigr]$ and $g(x)$ is $(\alpha_2,m_2)$-GA-concave
on $\bigl[0,\max\bigl\{a^{1/m_2},b\bigr\}\bigr]$, we have
\begin{align*}
&f\bigl(a^{1-t}b^t\bigr)\ge t^{\alpha_1}f(b)+m_1(1-t^{\alpha_1})f\bigl(a^{1/m_1}\bigr)
\end{align*}
and
\begin{align*}
&g\bigl(a^{1-t}b^t\bigr)\ge t^{\alpha_2}g(b)+m_2(1-t^{\alpha_2})g\bigl(a^{1/m_2}\bigr)
\end{align*}
for $t\in [0,1]$. Further letting $x=a^{1-t}b^t$ for $0\le t\le 1$ and utilizing H\"older's inequality reveal
\begin{align*}
&\quad\int_a^bf(x)g(x)\td x=(\ln b-\ln a) \int_0^1a^{1-t}b^tf\bigl(a^{1-t}b^t\bigr)g\bigl(a^{1-t}b^t\bigr)\td t \\
& \ge (\ln b-\ln a)\biggl\{\int_0^1a^{1-t}b^t\bigl[t^{\alpha _1}f(b) + m_1(1-t^{\alpha_1} )
f\bigl(a^{1/m_1}\bigr)\bigr]\\
&\quad\times\bigl[t^{\alpha_2 }g(b) + m_2(1-t^{\alpha_2 })g\bigl(a^{1/m_2}\bigr)\bigr]\td t\biggr\}\\
&=(\ln b-\ln a)\int_0^1a^{1-t}b^t\bigl[t^{\alpha_1+\alpha _2}
f(b)g(b)+m_1(1-t^{\alpha_1})t^{\alpha_2}f\bigl(a^{1/m_1}\bigr)g(b) \\
&\quad+ m_2 t^{\alpha_1}(1-t^{\alpha _2})f(b)g\bigl(a^{1/m_2}\bigr)
+ m_1m_2(1-t^{\alpha_1})(1-t^{\alpha_2})g\bigl(a^{1/m_2}\bigr)f\bigl(a^{1/m_1}\bigr)\bigl]\td t \\
& =(\ln b-\ln a)\bigl\{m_1m_2[L(a,b)-G(\alpha_1 ,1)-G(\alpha_2,1)\\
&\quad+G(\alpha_1+\alpha_2 ,1)]f\bigl(a^{1/m_1}\bigr)g\bigl(a^{1/m_2}\bigr)
 +m_1[G(\alpha_2,1)-G(\alpha_1+\alpha_2,1)]f\bigl(a^{1/m_1}\bigr)g(b)\\
&\quad+m_2[G(\alpha_1,1)-G(\alpha_1+\alpha_2,1)]f(b)g\bigl(a^{1/m_2}\bigr)+G(\alpha_1+\alpha_2,1)f(b)g(b)\bigr\}.
\end{align*}
The proof of Theorem~\ref{thm6-2012-GA-Ji} is complete.
\end{proof}

\begin{cor}\label{cor-3.6-2012-GA-Ji}
Under the conditions of Theorem~\ref{thm6-2012-GA-Ji}, if $\alpha_1=\alpha_2=m_1=m_2=1$, we have
\begin{multline}
\int_a^bf(x)g(x)\td x\ge(\ln b-\ln a)\{[2L(a,b)-a(\ln b -\ln a)-2a]f(a)g(a)+[a+b\\
-2L(a,b)][f(a)g(b)+f(b)g(a)]+[2L(a,b)+b(\ln b-\ln a)-2b]f(b)g(b)\}.
\end{multline}
\end{cor}

\end{document}